\documentclass[12pt]{article}
\usepackage{amssymb}
\usepackage{amsmath}
\usepackage{amsthm}
\usepackage[showonlyrefs]{mathtools}
\mathtoolsset{showonlyrefs}
\usepackage[colorlinks, linkcolor=blue, anchorcolor=blue, citecolor=blue]{hyperref}
\numberwithin{equation}{section}

\def\dim{\operatorname{dim}}
\def\diam{\operatorname{diam}}

\def\sing{\operatorname{sing}}
\def\modulus{\operatorname{mod}}
\def\N{\mathbb{N}}
\def\C{\mathbb{C}}
\def\CC{\widehat{\mathbb C}}
\def\R{\mathbb{R}}
\def\S{\mathcal{S}}
\def\B{\mathcal{B}}
\newtheorem{theorem}{Theorem}[section]
\newtheorem{lemma}{Lemma}[section]
\theoremstyle{remark}
\newtheorem{remark}{Remark}[section]
\begin{document}
\title{The Hausdorff dimension of Julia sets of meromorphic functions in the Speiser class}
\author{Walter Bergweiler and Weiwei Cui\thanks{W. Cui expresses his gratitude to the Centre for Mathematical Sciences of Lund University for providing a nice working environment.}}
\date{}
\maketitle
\begin{abstract}
We show that for each $d\in (0,2]$ there exists a meromorphic function $f$ such that 
the inverse function of $f$ has three singularities and the Julia set of $f$ has 
Hausdorff dimension $d$.
\end{abstract}

\section{Introduction and main results} \label{intro}
We will be concerned with the iteration of transcendental meromorphic functions
$f\colon\C\to\CC:=\C\cup\{\infty\}$. The main objects studied here
are the Fatou set $F(f)$, consisting of all  $z\in \mathbb C$ for which the
iterates  $f^k$   of $f$ are  defined and form  a normal family  in
some neighborhood of $z$, and the Julia set $J(f):=\CC\setminus F(f)$.
For an introduction to the dynamics of transcendental meromorphic functions we 
refer to~\cite{Bergweiler1993}.

The Speiser class $\S$ consists of all  transcendental meromorphic
functions $f$ for which the set $\sing(f^{-1})$ of singularities of the inverse,
i.e., the set of critical and asymptotic values of~$f$, is finite. 
More precisely, if $q$ denotes the cardinality of $\sing(f^{-1})$, then we write $f\in \S_q$.
Equivalently, $f\in \S$ if there exist finitely many points $a_1,\dots,a_q\in\CC$ such that 
\begin{equation} \label{3a}
f\colon \C\setminus f^{-1}(\{a_1,\dots,a_q\})\to \CC\setminus\{a_1,\dots,a_q\}
\end{equation} 
is a covering map. And if $q$ is the minimal number with this property, then $f\in\S_q$.
The monodromy theorem implies that we always have $q\geq 2$.

The Eremenko-Lyubich class $\B$ consists of all transcendental meromorphic functions $f$
for which $\sing(f^{-1})\setminus \{\infty\}$ is a bounded subset of~$\C$.
Thus
\begin{equation} \label{3b}
\S =\bigcup_{q=2}^\infty \S_q \subset \B.
\end{equation} 
Both the Speiser and the Eremenko-Lyubich class play an important role in 
transcendental dynamics.
The similarities and differences between these classes are addressed in 
a number of recent papers~\cite{AspenbergCui,Bishop2015,Bishop2015a,Bishop2017,Epstein2015}.
A survey of some of these as well as many other results concerning 
the dynamics of functions in $\S$ and $\B$ is given in~\cite{Sixsmith2018}.

Considerable attention has been paid to the Hausdorff dimension of Julia
sets; see~\cite{Kotus2008} and~\cite{Stallard2008} for surveys. 
We will denote the Hausdorff dimension of a subset $A$ of $\C$ by $\dim A$.

Baker~\cite{Baker1975} showed that if $f$ is a transcendental entire function, then
$J(f)$ contains continua so that $\dim J(f)\geq 1$.
Stallard~\cite[Theorem~1.1]{Stallard2000a} showed that for all $d\in(1,2)$ there exists
a transcendental entire function $f$ with $\dim J(f)=d$ while
Bishop~\cite{Bishop2018}
constructed an example with $\dim J(f)=1$.
Stallard's examples are actually in the Eremenko-Lyubich class~$\B$.
Previously she had shown that $\dim J(f)>1$ for entire $f\in\B$.
In particular, this is the case for entire functions in~$\S$.

Albrecht and Bishop~\cite{Albrecht2020}
showed that given $\delta>0$ there exists an entire function
$f\in \S$ such that $\dim J(f)<1+\delta$.
In fact, these were the first examples of entire functions in $\S$ for which the Julia set
has dimension strictly less than~$2$.
In their examples the inverse has three finite singularities.
Since every non-constant entire function has the asymptotic value $\infty$ by Iversen's theorem,
and since we include $\infty$ in $\sing(f^{-1})$, their examples are in~$\S_4$.

For transcendental meromorphic functions~$f$ we have
$\dim J(f)>0$ by a result of Stallard~\cite{Stallard1994}.
On the other hand, she showed~\cite[Theorem~5]{Stallard2000}
that for all $d\in (0,1)$ there exists a  transcendental meromorphic
function $f$ such that $\dim J(f)=d$. Again her examples are in~$\B$.
Together with her result covering the interval $(1,2)$ mentioned
above, and since $J(\exp z)=\C$ by a result of 
Misiurewicz~\cite{Misiurewicz1981} and $J(\tan z)=\R$,
it follows 
that for all $d\in (0,2]$ there exists $f\in\B$ such that $\dim J(f)=d$.

We shall show that such examples also exist in the Speiser class $\S$ and in fact in~$\S_3$.

\begin{theorem}\label{theorem1}
Let $d\in (0,2]$. Then there exists a function $f\in \S_3$ such that $\dim J(f)=d$.
\end{theorem}

We note that if $f\in\S_2$, then $\dim J(f)>1/2$; see~\cite[Theorem 3.11]{Baranski1995}
and~\cite[Remark~3.2 and Section~4.3]{Mayer2009}.
On the other hand, Bara\'nski~\cite[Section~4]{Baranski1995} showed that 
for  $f_\lambda(z)=\lambda\tan z\in\S_2$ the function $\lambda\mapsto \dim J(f_\lambda)$
maps $(0,1]$ monotonically and continuously onto $(1/2,1]$.
It seems likely that for $d\in(1,2]$ there also exists $f\in\S_2$ such that $\dim J(f)=d$.
The main interest of Theorem~\ref{theorem1} thus lies in the case where
$0<d\leq 1/2$.

\section{Preliminary results} \label{prelim}
The escaping set $I(f)$ of a meromorphic function $f$ is defined 
as the set of all $z\in\C$ for which $f^n(z) \to \infty$ as $n\to\infty$.
We always have $I(f)\neq\emptyset$ and $J(f)=\partial I(f)$.
This was shown by Eremenko~\cite{Eremenko1989} for transcendental entire $f$ and
by Dom\'inguez~\cite{Dominguez1998} for transcendental mero\-morphic $f$.

For $f\in\B$ we have $I(f)\subset J(f)$.
This is due to Eremenko and Lyubich~\cite[Theorem~1]{Eremenko1992} for entire $f$ and 
to Rippon and Stallard~\cite[Theorem~A]{Rippon1999} for meromorphic~$f$.

Theorem~\ref{theorem1} complements the result of~\cite{AspenbergCui} where it was
shown that for all $d\in[0,2]$ there exists a meromorphic function $f\in\S$ such
that $\dim I(f)=d$.

An important concept in the theory of meromorphic functions is 
the order; see, e.g.,~\cite{Hayman1964}.
The following result was proved in~\cite[Theorem~1.1]{Bergweiler2012}.
\begin{lemma}\label{lemma5}
Let $f\in \B$ be of finite order~$\rho$.
Suppose that $\infty$ is not an asymptotic value of $f$ and that
there exists $M\in\N$ such that all but finitely many poles of $f$ have multiplicity 
at most~$M$. Then
\begin{equation} \label{1b}
\dim I(f)\leq \frac{2M\rho}{2+M\rho}.
\end{equation} 
\end{lemma}
We will not actually use this lemma. Instead we will use the following closely related
result, which can be proved by the same method.
\begin{lemma}\label{lemma1}
Let $f$ be as in the Lemma \ref{1b}. Suppose, furthermore, that $f(0)\neq 0$.
For $\lambda\in\C\setminus\{0\}$ define $f_\lambda(z)=f(\lambda z)$. Then
\begin{equation} \label{1a0}
\limsup_{\lambda\to 0} \dim J(f_\lambda) \leq \frac{2M\rho}{2+M\rho}.
\end{equation} 
\end{lemma}
\begin{proof}
We proceed as in~\cite{Bergweiler2012}.
Let $(a_j)$ be the sequence  of poles of $f$ and let $m_j$  be the   multiplicity of~$a_j$.
Let $b_j\in \C \setminus\{0\}$ be such that 
\begin{equation} \label{5a}
f(z)\sim\left(\frac{b_j}{z-a_j}\right)^{m_j} \quad \text{as}\ z\to a_j .
\end{equation} 
Let 
\begin{equation} \label{5b}
t>  \frac{2M \rho}{2+ M \rho}.
\end{equation} 
By~\cite[Lemma~3.1]{Bergweiler2012} we have 
\begin{equation} \label{5c}
\sum_{j=1}^\infty\left(\frac{|b_j|}{|a_j|^{1+1/M}}\right)^t < \infty .
\end{equation} 
Let $a_j^\lambda$ and $b_j^\lambda$ the corresponding values for~$f_\lambda$.
Then $a_j^\lambda=a_j/\lambda$ and $b_j^\lambda=b_j/\lambda$  so that
\begin{equation} \label{5d}
\sum_{j=1}^\infty\left(\frac{|b_j^\lambda|}{|a_j^\lambda|^{1+1/M}}\right)^t 
=|\lambda|^{t/M}\sum_{j=1}^\infty\left(\frac{|b_j|}{|a_j|^{1+1/M}}\right)^t .
\end{equation} 
As in~\cite{Bergweiler2012} we choose $R_0>|f(0)|$ such that 
$\sing(f_\lambda^{-1})=\sing(f^{-1})\subset D(0,R_0)$ and $R\geq 2^MR_0$. 
Here and in the following $D(a,r)$ denotes the open disk of radius $r$ around
a point $a\in\C$.
Choosing $\lambda$ small we can achieve that $f(D(0,3R))\subset D(0,R)$ 
so that $F(f_\lambda)\subset D(0,3R)$.

We put $B(R):=\{z\colon |z|>R\}\cup \{\infty\}$ and,
as in~\cite[pp.~5376f.]{Bergweiler2012},
consider the collection $E_l$ of all components $V$ of $f^{-l}(B(R))$ 
for which $f^k(V)\subset B(3R)$ for $0\leq k\leq l-1$.
Then $E_l$ is a cover of $J(f_\lambda)$ and we have
\begin{equation} \label{5f}
\sum_{V\in E_l}\!\left(\diam_{\chi}(V)\right)^t
 = \frac{1}{M}\!\left(\frac{32}{(2R)^{1/M}24}\right)^{\! t}
\left( M (2^{1/M} 24)^{t}
\sum_{j=1}^\infty
\!\left(\frac{|b_j^\lambda|}{|a_j^\lambda|^{1+1/M}}\right)^{\! t} \right)^{\! l}.
\end{equation} 
Here $\diam_{\chi}(V)$ denotes the spherical diameter of~$V$.
Together with~\eqref{5c} and~\eqref{5d} the last estimate yields that if $\lambda$
is sufficiently small, then 
\begin{equation} \label{5g}
\sum_{V\in E_l}\left(\diam_{\chi}(V)\right)^t\to  0
\quad\text{as}\ l\to\infty.
\end{equation} 
Hence $\dim J(f_\lambda)\leq t$ for small~$\lambda$.
\end{proof}

It follows from a recent result of Mayer and Urba\'nski~\cite{Mayer2021} that 
Lemma~\ref{lemma1} and~\eqref{1b} can be sharpened to 
\begin{equation} \label{1a}
\lim_{\lambda\to 0} \dim J(f_\lambda)=\dim I(f).
\end{equation} 
In fact, their result says that
 $\dim I(f)$ is the infimum of the set of all $t>0$ for which~\eqref{5c} holds.

For entire functions, the following result can be found in~\cite[Section~3]{Eremenko1992}
and~\cite[Proposition~2.3]{Epstein2015}.
\begin{lemma}\label{lemma2}
Let $f,g\in\S_3$ and suppose that  there exist homeomorphisms
$\psi\colon \CC\to\CC$ and $\phi\colon\C\to\C$
such that $\psi\circ f=g\circ\phi$.
Then there exist a fractional linear transformation $\alpha\colon\CC\to\CC$
and an affine map $\beta\colon \C\to\C$ such that
$\alpha\circ f=g\circ \beta$.
\end{lemma}

\begin{proof}
Since $f\in\S_3$ it follows from \cite[Observation 1.10]{Epstein2015} that
there exists a fractional linear transformation $\alpha:\CC\to\CC$ which is isotopic 
to $\psi$ relative to $\sing(f^{-1})$. 
We follow the argument in the proof of \cite[Proposition 2.3(a)]{Epstein2015}. 
Let $(\psi_t)_{t\in [0,1]}$ be the isotopy between $\psi_0=\psi$ and $\psi_1=\alpha$. 
By the isotopy lifting property, there exists a unique isotopy $(\phi_t)_{t\in[0,1]}$ in
\begin{equation} \label{1c}
U:=f^{-1}\!\left(\CC\setminus\sing(f^{-1})\right)
\end{equation} 
such that $\phi_0=\phi$ and $\psi_t\circ f=g\circ \phi_t$ in $U$ for all $t\in[0,1]$.
It remains to show 
that $\phi_t$ extends continuously to the preimages of the singular values of $f$ and 
coincides with $\phi$ there for all $t$.

Let $z_0\in\C$ be a preimage of $v_0\in\sing(f^{-1})$. We have to show
that $\phi_t(z)\to \phi(z_0)$
as $z\to z_0$. We take a small neighborhood $D$ of $z_0$ such that
$f\colon D\to f(D)$ is a proper map, with no critical points except possibly at $z_0$.
We may also assume that $f(z)\neq v_0$ for $z\in\overline{D}\setminus \{z_0\}$.
It then suffices to show that $\phi_t(z)\in\phi(D)$ if $z$ is sufficiently close to $z_0$.
By the continuity of $f$ and the properties of an isotopy, we have
\begin{equation} \label{1c1}
\psi_t(f(z))\to\psi_t(f(z_0))=\psi(v_0)\quad\text{as}\ z\to z_0,
\end{equation} 
uniformly in $t\in[0,1]$.
It follows that
\begin{equation} \label{1c2}
g(\phi_t(z))=\psi_t(f(z))\in\psi(f(D))=g(\phi(D))
\end{equation} 
if $z$ is sufficiently close to~$z_0$.
Thus $\phi_t(z)$ can be obtained by analytic continuation of $g^{-1}$ along the 
curve $t\mapsto \psi_t(f(z))$. It follows that $\phi_t(z)\in\phi(D)$ if $z$ is 
sufficiently close to $z_0$.

With $\beta:=\phi_1$ we have $\alpha\circ f=g\circ \beta$.
If $z$ is not a critical point of $f$, then $\beta$ is of the
form $g^{-1}\circ \psi\circ f$ near $z$ for some branch  of the inverse of $g$. 
Hence $\beta$ is holomorphic.
Since $\beta\colon\C\to\C$ is a homeomorphism, this implies that $\beta$ is affine.
\end{proof}

The following result is due to 
Kotus and Urba\'nski~\cite{Kotus2003}.
\begin{lemma}\label{lemma3}
Let $f$ be an elliptic function. Let $q$ be the maximal multiplicity 
of the poles of~$f$. Then
\begin{equation} \label{1d}
\dim J(f)> \frac{2q}{q+1}.
\end{equation} 
\end{lemma}
We shall need a variation of this result.
\begin{lemma}\label{lemma3a}
Let $f$ be as in Lemma~\ref{lemma3}.
If $(f_n)$ is a sequence of meromorphic functions which converges locally
uniformly to~$f$, then
\begin{equation} \label{1e}
\dim J(f_n)> \frac{2q}{q+1}
\end{equation} 
for all large~$n$.
\end{lemma}
Lemma~\ref{lemma3a} can be deduced from the work of
Kotus and Urba\'nski; see Remark~\ref{remark1} below.
But for the convenience of the reader we will include a proof of Lemma~\ref{lemma3a},
thereby reproving Lemma~\ref{lemma3}.
Here we will use the following result~\cite[Proposition~9.7]{Falconer1990}.
\begin{lemma} \label{lemma4a}
Let $S_1,\dots,S_m$ be contractions on a closed subset $K$ of $\R^d$ such that
there exists $b_1,\dots,b_m\in (0,1)$ with
\begin{equation}\label{4b}
b_k|u-v|\leq |S_k(u)-S_k(v)|
\quad \text{for } u,v\in K 
\end{equation}
and $1\leq k\leq m$.
Suppose that $K_0$ is a non-empty compact subset of $K$ with
\begin{equation}\label{4c}
K_0=\bigcup_{k=1}^m S_k(K_0)
\end{equation}
and $S_j(K_0)\cap S_k(K_0)=\emptyset$ for $j\neq k$.
Let $t>0$ with
\begin{equation}\label{4d}
\sum_{k=1}^m b_k^t=1.
\end{equation}
Then $\dim K_0\geq t$.
\end{lemma}
\begin{proof}[Proof of Lemmas~\ref{lemma3} and~\ref{lemma3a}]
Let $(a_j)$ be the sequence of poles of multiplicity~$q$. 
For sufficiently large $R$ there is a neighborhood $U_j$ of $a_j$ such that
\begin{equation}\label{2a}
f\colon U_j\setminus\{a_j\}\to\{z\in\C\colon |z|>R\}
\end{equation}
is a covering map of degree~$q$. 
There exists $r_1>0$ such that $D(a_j,r_1)\subset U_j$ for all~$j$. Let $0<r_0<r_1$.
Then there exists $M\in\N$ such that 
\begin{equation}\label{2a1}
\overline{D(a_k,r_0)}\subset \overline{U_k}\subset f(D(a_j,r_0))
\end{equation}
  for all $j,k\geq M$.
Thus for $j,k\geq M$ there exists $V_{j,k}\subset D(a_j,r_0)$ 
such that $f\colon V_{j,k}\to D(a_k,r_0)$ is biholomorphic.

Let $W_k:=f^{-1}(V_{k,M})\cap D(a_M,r_0)$. 
Then 
\begin{equation}\label{2b}
f^2\colon W_{k}\to D(a_M,r_0)
\end{equation}
is biholomorphic. Moreover, $f^2$ extends to a bijective map from 
$\overline{W_k}$ to $K:=\overline{D(a_M,r_0)}$.
Let $S_k\colon K\to  \overline{W_{k}}$ be the inverse function
of $f^2\colon \overline{W_{k}}\to K$.
Then $S_k$ extends to an injective map $S_k\colon D(a_M,r_1)\to\C$.
Choosing $r_0\leq (2-\sqrt{3})r_1$ we conclude that $W_k=S_k(D(a_M,r_0))$ 
is convex; see~\cite[Theorem~2.13]{Duren1993}.

In order to apply Lemma~\ref{lemma4a} we note that
since $f$ has a pole of multiplicity $q$ at~$a_j$, 
there exists $c_1>0$ such that 
\begin{equation}\label{2c}
|f'(z)|\leq c_1 |f(z)|^{(q+1)/q}
\quad\text{for}\ z\in D(a_j,r_0).
\end{equation}
This implies that there exists $c_2>0$ such that
\begin{equation}\label{2d}
|f'(z)|\leq c_2 |a_k|^{(q+1)/q} 
\quad\text{if}\ z\in D(a_j,r_0) \ \text{and}\ f(z)\in D(a_k,r_0).
\end{equation}
With $c_3:= c_2^2 |a_M|^{(q+1)/q}$ this yields that
\begin{equation}\label{2e}
|(f^2)'(z)|=|f'(f(z))f'(z)|\leq c_3 |a_k|^{(q+1)/q} 
\quad\text{for}\ z\in W_k.
\end{equation}
Since $W_k$ is convex this yields that
\begin{equation}\label{2e1}
|u-v|=|f^2(S_k(u))-f^2(S_k(v))|\leq  c_3 |a_k|^{(q+1)/q}|S_k(u)-S_k(v)|
\quad\text{for}\ u,v\in K.
\end{equation}
It follows that~\eqref{4b} holds with
\begin{equation}\label{2f}
b_k:=\frac{1}{c_3 |a_k|^{(q+1)/q}} .
\end{equation}
Let now $N\geq M$ and define $t>0$ by 
\begin{equation}\label{2g}
\sum_{k=M}^N b_k^t=1.
\end{equation}
It follows from Lemma~\ref{lemma4a} that the limit set of the iterated function system
$\{S_k\colon M\leq k\leq N\}$ has Hausdorff dimension at least~$t$. 
It is easily seen that this limit set is contained in the Julia set.
Thus $\dim J(f)\geq t$.
Since 
\begin{equation}\label{2h}
\sum_{k=M}^\infty \frac{1}{|a_k|^2}=\infty 
\end{equation}
we have $t>2q/(q+1)$ if $N$ is large enough.
This yields Lemma~\ref{lemma3}.

To prove Lemma~\ref{lemma3a} we note that for large $n$ there are
domains $W_k^n$ and $U_M^n$ close to $W_k$ and $U_M$ such that 
\begin{equation}\label{2i}
f_n^2\colon W_k^n\to U_M^n
\end{equation}
is biholomorphic, and the inverse function $S_k^n$ satisfies~\eqref{4b} with
a constant $b_k^n$ instead of $b_k$, with $b_k^n\to b_k$ as $n\to\infty$.
The conclusion then follows again from Lemma~\ref{lemma4a}.
\end{proof}
\begin{remark}\label{remark1}
The argument of Kotus and Urba\'nski~\cite{Kotus2003} is similar to the one used above,
but they use an infinite iterated function system and apply results of 
Mauldin and Urba\'nski~\cite{Mauldin1996} concerning such systems.
However, it would suffice to consider a sufficiently large finite subsystem.
This would yield Lemma~\ref{lemma3a} as above.

The proof actually yields that the hyperbolic dimension of $f$ and $f_n$, and 
not only the Hausdorff dimension of their Julia sets, have the given lower bound;
see~\cite{Baranski2009,Rempe2009} for a discussion of the hyperbolic dimension
of meromorphic functions.
\end{remark}

\begin{lemma}\label{lemma6}
Let $f\in \S$. Suppose that $f$ has an attracting fixed point whose
immediate attracting basin contains all finite singularities of~$f^{-1}$.
Suppose also that $\infty$ is not an asymptotic value of~$f$ and that there exists
a uniform bound on the  multiplicities of the poles of~$f$.
Then $J(f)$ is totally disconnected and
$F(f)$ is connected.
\end{lemma}
There are several closely related results in the literature,
see~\cite[Theorem~G]{Baker2001}, \cite[Theorem~A and Corollary~ 3.2]{Dominguez2014}
and~\cite[Theorem~2.7]{Zheng2015}.
However, none of these results seems to apply exactly to the situation we have.

Zheng~\cite[Theorem~2.7]{Zheng2015} showed that the conclusion 
of Lemma~\ref{lemma6} holds if $\infty\notin\sing(f^{-1})$. 
Thus his result would apply if the poles are assumed to be simple, while we
 allow multiple poles.
Hawkins and Koss (\cite[Theorem~3.12]{Hawkins2005};
see also~\cite[Theorem~3.2]{Hawkins2019})
do not require that the poles are simple,
but they restrict to elliptic functions.

Our proof of Lemma~\ref{lemma6} will use some ideas from the papers mentioned.
A difference to the methods employed there, however, is that 
we will use the following consequence of the Gr\"otzsch
inequality; see\cite[Section~5.2]{Branner1992} and~\cite[Corollary~A.7]{Milnor2000}.
Here and in the following $\modulus(A)$ denotes the modulus of an annulus~$A$.
\begin{lemma}\label{lemma7}
Let $(G_k)$ be a sequence of simply connected domains in $\C$ such that
$A_k:=G_k\setminus \overline{G_{k+1}}$ is an annulus for all $k\in\N$.
Suppose that
\begin{equation} \label{7a1}
\sum_{k=1}^\infty \modulus(A_k) =\infty.
\end{equation}
Then
$\bigcap_{k=1}^\infty G_k$
consists of a single point.
\end{lemma}

\begin{proof}[Proof of Lemma~\ref{lemma6}]
Let $\xi$ be the attracting fixed point whose attracting basin $W$ contains
all finite singularities of~$f^{-1}$.
Then the postsingular set
\begin{equation} \label{7a}
P(f):=\overline{\bigcup_{n=0}^\infty f^n\!\left( \sing(f^{-1})\setminus\{\infty\}\right)}
\end{equation}
is a compact subset of~$W$. It is not difficult to see that there
exist Jordan domains $U$ and $V$ such that
\begin{equation} \label{7b}
\{\xi\}\cup P(f)\subset U\subset \overline{U}\subset V \subset \overline{V} \subset W.
\end{equation}
Then $A:=V\setminus\overline{U}$ is an annulus. Clearly, $A\subset W$.

Let $(a_j)$ be the sequence of poles of $f$ and let $m_j$ denote the multiplicity of~$a_j$.
By hypothesis, there exists $M\in\N$ such that $m_j\leq M$ for all~$j$.
Let $Y_j$ be the component of the preimage of $\CC\setminus \overline{U}$ that contains~$a_j$.
Then $f\colon Y_j\setminus\{a_j\}\to \C\setminus \overline{U}$ is a covering of degree~$m_j$.
Putting $B_j:=f^{-1}(A)\cap Y_j$ we find that $B_j$ is an annulus with
\begin{equation} \label{7c}
\modulus(B_j)=\frac{1}{m_j}\modulus(A)\geq\frac{1}{M}\modulus(A).
\end{equation}
To prove that $J(f)$ is totally disconnected, let $z\in J(f)$.
We want to show that the component of $J(f)$ containing $z$ consists of the point $z$ only.
First we note that
for all $n\in\N$ there exists $j(n)\in\N$ such that $f^n(z)\in Y_{j(n)}$.
Let $X_n$ be the component of $f^{-n}(Y_{j(n)})$ containing~$z$.
Note that $Y_{j(n)}$ and $X_n$ are Jordan domains.
Since $\partial U$ and hence $\partial X_n$ are contained in $F(f)$,
the component of $J(f)$ containing $z$ is contained in the intersection of the~$X_n$.
It thus suffices to prove that this intersection consists of only one point.

In order to do so we note that since
$P(f)\subset U$, the map $f^{n}\colon X_n\to Y_{j(n)}$ is biholomorphic.
This implies that
\begin{equation} \label{7d}
C_n:=f^{-n}(B_{j(n)})\cap X_n= f^{-n-1}(A)\cap X_n
\end{equation}
is an annulus satisfying
\begin{equation} \label{7e}
\modulus(C_n)=\modulus(B_{j(n)})\geq \frac{1}{M}\modulus(A).
\end{equation}
Moreover, $X_n$ is equal to the union of $C_n$ and the component
of $\C\setminus C_n$ that contains~$z$.

Since the closure of $A$ is a compact subset of the attracting basin $W$ of~$\xi$,
there exists $p\in\N$ such that $f^p(A)\subset U$.
In particular,
\begin{equation} \label{7f}
f^p(A)\cap A=\emptyset.
\end{equation}
This implies that $C_{n+p}\cap C_n=\emptyset$ for all~$n\in\N$.
In fact, if $w\in C_{n+p}\cap C_n$, then $f^{n+p+1}(w)\in f^p(A)\cap A$ by~\eqref{7d},
contradicting~\eqref{7f}.
It follows that
\begin{equation} \label{7g}
X_{n+p}\subset X_n\setminus C_n,
\end{equation}
and hence that
\begin{equation} \label{7h}
\modulus\!\left(X_n\setminus\overline{X_{n+p}}\right)
\geq \modulus(C_n)\geq \frac{1}{M}\modulus(A).
\end{equation}
As already mentioned above it now follows from
Lemma~\ref{lemma7}, applied with $G_k=X_{1+pk}$, that
$J(f)$ is totally disconnected. Of course, this yields that $F(f)$ is connected.
\end{proof}

\section{Proof of Theorem~\ref{theorem1}} \label{prooftheorem}
Let $G$ be the conformal map from the triangle with 
vertices $0$, $\pi/2$ and $i\pi/2$ onto the lower half-plane such that
\begin{equation} \label{1g0}
G(\pi/2)=0, \quad G(0)=1
\quad\text{and}\quad
G(i\pi/2)=\infty.
\end{equation} 
Extending this to the whole 
plane by reflections we obtain an elliptic function~$G$. 
The critical values of $G$ are $0$, $1$ and $\infty$ so that $G\in\S_3$.
The zeros and poles of $G$ have multiplicity $4$ while the $1$-points have multiplicity~$2$.
We also note that $G(\R)=[0,1]$.

We may express  $G$ in terms of the Weierstrass $\wp$-function 
with periods $\pi$ and $\pi i$. 
The critical values of $\wp$ are $e_1$, $e_2$, $e_3$ and $\infty$, with 
\begin{equation} \label{1g}
e_2=\wp((\pi+i \pi)/2)=0
\quad\text{and}\quad
e_1=\wp(\pi/2)=-e_3=-\wp(i\pi/2) .
\end{equation} 
It follows from this that
\begin{equation} \label{1h}
G(z)=\left(\frac{\wp(z+i\pi/2)}{e_1}\right)^2 .
\end{equation} 

First we construct an example of a function $f\in\S_3$ where the Julia set 
is the whole sphere
and thus has Hausdorff dimension~$2$.
To this end we consider the function
\begin{equation} \label{1h1}
f(z):=i G(\pi z/2).
\end{equation} 
Then $f\in\S_3$.
The critical values of $f$ are $0$, $i$ and $\infty$ and we have
$f(0)=i$ and $f(i)=\infty$.
To prove that $J(f)=\CC$ we note that $f$ has no wandering domains~\cite{Baker1992}
and no Baker domains~\cite[Corollary to Theorem~A]{Rippon1999}.
All other types of components of $F(f)$ are related to the singularities of~$f^{-1}$;
see~\cite[Theorem~7]{Bergweiler1993} for the exact statement.
Since the points in $\sing(f^{-1})\cap\C=\{0,i\}$ are mapped to $\infty$ by $f$ or $f^2$ 
this yields that $F(f)=\emptyset$ and hence $J(f)=\CC$ as claimed.

To construct functions in $\S_3$
for which the Julia set has Hausdorff dimension $d\in (0,2)$, 
we consider, for $p\in\N$ and small $\eta\in (0,\pi/2)$, the function $H$
defined by $H(z):=\eta G(z)^p$.
The critical values of $H$ are $0$, $\eta$ and~$\infty$. Thus $H\in\S_3$.
We have $H(0)=\eta$ and $H(\pi/2)=0$. Also, $H$ decreases in the interval $(0,\pi/2)$.
Choosing $\eta$ sufficiently small we can achieve that $H$ has an attracting
fixed point $\xi\in (0,\pi/2)$ whose attracting basin contains $[0,\eta]$
and thus, since $H(\R)=[0,\eta]$, also contains~$\R$.
Lemma~\ref{lemma6} implies that this attracting basin coincides with $F(H)$ and that
$J(H)$ is totally disconnected.

Since $H((0,\pi/2))=(0,\eta)$ we actually have $\xi\in (0,\eta)$.
Choosing $\eta$ small we can also achieve that $H''(z)\neq 0$ for $0<|z|\leq \eta$.

Let now $m$ be a (large) odd integer.
Then 
\begin{equation} \label{1j}
h_m(z)=H(m\arcsin(z/m)) .
\end{equation} 
defines a meromorphic function $h_m\in \S_3$.
Similar examples were already considered by Teichm\"uller~\cite[p.~734]{Teichmueller1944},
and later by Bank and Kaufman~\cite[Section~5]{Bank1976},
Langley~\cite[Section~2]{Langley2002}
and Eremenko~\cite{Eremenko2004}.

The elliptic function $H(z)=h_m(m\sin(z/m))$ has order~$2$.
A result of Edrei and Fuchs~\cite[Corollary~ 1.2]{Edrei1964} thus
yields that $h_m$ has order~$0$.
In fact, as in the papers cited above we find that there exists a constant $c$ such
that the Nevanlinna characteristic satisfies $T(r,h_m)\sim c( \log r)^2$ as $r\to\infty$.

For large $m$ the function $h_m$ has an attracting fixed point $\xi_m$, with 
$\xi_m\to \xi$ as $m\to\infty$,
such that the attracting basin of $\xi_m$ contains the interval
$[0,\eta]$ and hence~$\R$.
Lemma~\ref{lemma6} implies that this attracting basin is connected and coincides with $F(h_m)$.
Choosing $m$ large we can also achieve that $h_m$ decreases in the interval $[0,\eta]$
and that $h_m''(z)\neq 0$ for $0<|z|\leq \eta$.
This implies that $h_m'$ decreases in the interval $[0,\eta]$.

The poles of $H$ and $h_m$ have multiplicity $4p$.
Except for the zeros at $\pm m$, which have multiplicity $2p$, 
the zeros of $h_m$ also have multiplicity $4p$.
The $\eta$-points on the real axis have multiplicity~$2$, but if $p\geq 2$, then 
$H$ and $h_m$ also have simple $\eta$-points (corresponding to the points where 
$G$ takes the $p$-th roots of unity).

As the poles of $H$ have multiplicity $4p$, Lemma~\ref{lemma3} yields that 
\begin{equation} \label{1i}
\dim J(H) > \frac{8p}{4p+1} .
\end{equation} 
Moreover, this lemma says that if $m$ is sufficiently large, then
\begin{equation} \label{1k}
\dim J(h_m)> \frac{8p}{4p+1}.
\end{equation} 
We fix such a value of $m$ and, for $\lambda\in (0,1]$, we
put $f_\lambda(z):=h_m(\lambda z)$ so that $f_1=h_m$.

Since $h_m$ has order~$0$, 
Lemma~\ref{lemma1} yields that 
\begin{equation} \label{1l}
\lim_{\lambda\to 0}\dim J(f_\lambda) = 0.
\end{equation} 

We will show that the function $\lambda\mapsto \dim J(f_\lambda)$ is continuous
in the interval $(0,1]$. Since $f_1=h_m$ it then follows from~\eqref{1k} and~\eqref{1l}
that for all $d\in (0,8p/(4p+1)]$ there exists
$\lambda\in (0,1]$ such that $\dim J(f_\lambda)=d$. Since $p$ can be chosen arbitrarily
large, this yields the conclusion. 

It remains to prove that $\lambda\mapsto \dim J(f_\lambda)$ is continuous.
Recalling that $h_m$ and $h_m'$ are decreasing in the interval $[0,\eta]$
we can deduce that $f_\lambda$ has an attracting fixed point 
$\zeta_\lambda\in(0,\eta)$ and that the multiplier
\begin{equation} \label{1l1}
m_\lambda:=f_\lambda'(\zeta_\lambda)=\lambda h_m'(\lambda \zeta_\lambda)
\end{equation} 
is a decreasing function of $\lambda$ in the interval $(0,1]$.
Note here that $\zeta_1=\xi_m$ and $m_1=h_m'(\xi_m)<0$.
As before it follows from Lemma~\ref{lemma6} that 
the attracting basin of $\zeta_\lambda$ is connected and coincides with the Fatou set
of $f_\lambda$.

Let $\lambda\in (0,1]$.
K\oe nigs' theorem~\cite[Theorem~8.2]{Milnor1999} yields that 
there exists a function $g$ holomorphic and injective 
in some neighborhood $U$ of $\zeta_\lambda$ 
such that $g(\zeta_\lambda)=0$, $g'(\zeta_\lambda)=1$ and 
\begin{equation} \label{1n}
g(f_\lambda(g^{-1}(z)))=m_\lambda z 
\end{equation} 
for all $z\in g(U)$.
For $\kappa\in(0,1]$ we put
\begin{equation} \label{1o}
\gamma=\frac{\log m_\kappa}{\log m_\lambda} -1
\end{equation} 
and define $h\colon\C\to\C$,
$h(z)= z|z|^\gamma$. Then 
\begin{equation} \label{1p}
h(m_\lambda h^{-1}(z))=m_\kappa z. 
\end{equation} 
With $\phi=h\circ g\colon U\to\C$ we then have
\begin{equation} \label{1q}
\phi(f_\lambda(\phi^{-1}(z)))=m_\kappa z
\end{equation} 
for $z\in \phi(U)$. The maps $h$ and $\phi$ are $K$-quasi\-con\-for\-mal with 
\begin{equation} \label{1r}
K= \max\left\{ \frac{\log m_\kappa}{\log m_\lambda}, \frac{\log m_\lambda}{\log m_\kappa}\right\}. 
\end{equation} 
For a detailed account of quasiconformal mappings, we refer to \cite{Lehto1973}. The complex dilatation $\mu(z):=\phi_{\overline{z}}(z)/\phi_z(z)$ satisfies 
\begin{equation} \label{1s}
\mu(f_\lambda(z))= \mu(z) \frac{f_\lambda'(z)}{\;\overline{f_\lambda'(z)}\;}
\end{equation} 
if $z,f_\lambda(z)\in U$.
We may use~\eqref{1s} to extend $\mu$ to~$\C$. More precisely, we put
$\mu(z)=0$ for $z\in J(f_\lambda)$ while for $z\in F(f_\lambda)$ we define
\begin{equation} \label{1t}
 \mu(z) = \mu(f_\lambda^n(z)) \frac{\;\overline{(f_\lambda^n)'(z)}\;}{(f_\lambda^n)'(z)}, 
\end{equation} 
where $n$ is chosen so large that $f_\lambda^n(z)\in U$.
Using~\eqref{1s} it is easily seen that $\mu$ is well-defined, i.e., the definition does not
depend on the value of~$n$ chosen in~\eqref{1t}. We find that~\eqref{1s} holds for all~$z$.

Let $\psi\colon\C\to\C$ be the solution of the Beltrami equation 
\begin{equation} \label{1t1}
\mu(z)=\frac{\psi_{\overline{z}}(z)}{\psi_z(z)}, 
\end{equation} 
normalized by $\psi(0)=0$ and $\psi(\eta)=\eta$.
It follows from~\eqref{1s} that 
\begin{equation} \label{1t2}
k:=\psi\circ f_\lambda\circ \psi^{-1}
\end{equation} 
is meromorphic.
Since $f_\lambda$ is symmetric with respect to the real axis, the same applies 
to $g$, $\phi$, $\mu$, $\psi$ and~$k$.
By definition, $f_\lambda$ is even.
In order to show that $k$ is also even, 
we first note that since $f_\lambda$ is even, it follows from~\eqref{1s} that $\mu$ is even.
This implies that $\psi(z)$ and $\psi(-z)$ have the same complex dilatation.
Hence there exists an affine map $L$ such that $\psi(z)=L(\psi(-z))$.
Since $\psi(0)=0$ we have $L(0)=0$ so that $L$ has the form $L(z)=az$
for some $a\in\C\setminus\{0\}$. We also see that $L$ is real on the 
real axis so that $a\in\R\setminus\{0\}$.
Since $a\psi(i)=L(\psi(i))=\psi(-i)=\overline{\psi(i)}$ we find that $|a|=1$.
As $\psi$ is injective this implies that $a=-1$ so that $\psi$ is odd. Hence $k$ is even.

Since the complex dilatations of $\phi$ and $\psi$ agree in~$U$, we have 
$\phi=\tau\circ\psi$ for some function $\tau$ holomorphic and injective on $\psi(U)$.
Together with~\eqref{1q} this implies that 
\begin{equation} \label{1u}
k(z)= \psi(f_\lambda(\psi^{-1}(z)))=\tau^{-1}(m_\kappa\tau(z)) \sim m_\kappa z
\end{equation} 
as $z\to \tau^{-1}(0)=\psi(\zeta_\lambda)$. Thus $k'(\psi(\zeta_\lambda))=m_\kappa$. 
Moreover, $k(\psi(\zeta_\lambda))=\psi(\zeta_\lambda)$.
In other words, $\psi(\zeta_\lambda)$ is a fixed point of $k$ of multiplier~$m_\kappa$.

Another function with a fixed point of multiplier~$m_\kappa$ is~$f_\kappa$.
We will show that $k=f_\kappa$.
In order to do so we note that both $k$ and $f_\kappa$ are in $\S_3$,
with critical values $0$, $\eta$ and~$\infty$.
It follows from Lemma~\ref{lemma2} and~\eqref{1t2} that there exist 
a fractional linear transformation $\alpha$ and an affine map $\beta$ such that
$\alpha\circ k=f_\lambda\circ \beta$.
Since all poles of $k$ and $f_\kappa$ have multiplicity~$4p$,
all but two zeros of both functions have multiplicity~$4p$,
and all $\eta$-points of both functions have multiplicity $2$ or~$1$,
we find that $\alpha(0)=0$, $\alpha(\eta)=\eta$ and $\alpha(\infty)=\infty$.
Hence $\alpha(z)\equiv z$ so that $k=f_\lambda\circ \beta$.

As $\beta$ is affine we have $-\beta(-z)=\beta(z)-2\beta(0)$.
Noting that $k$ and $f_\lambda$ are even we deduce that 
\begin{equation} \label{1v0}
f_\lambda(\beta(z)-2\beta(0))= f_\lambda(-\beta(-z)) 
= f_\lambda(\beta(-z))=k(-z)=k(z)=f_\lambda(\beta(z)).
\end{equation} 
Since periodic functions have order at least $1$ while $f_\lambda$ has order~$0$,
this implies 
that $\beta(0)=0$ so that $\beta$ has the form $\beta(z)=cz$ for some 
constant~$c$. Thus $k(z)=f_\lambda(cz)=h_m(\lambda cz)$.
As $k$ has an attracting fixed point of multiplier $m_\kappa$
this yields that $c=\kappa/\lambda$ so that $k(z)=h_m(\kappa z)=f_\kappa(z)$.

Inserting $k=f_\kappa$ in~\eqref{1t2}, we obtain
\begin{equation} \label{1v}
f_\kappa=\psi\circ f_\lambda\circ \psi^{-1}.
\end{equation} 
This implies that 
\begin{equation} \label{1w}
J(f_\kappa)= \psi(J(f_\lambda)).
\end{equation} 
As $\psi$ is $K$-quasi\-con\-for\-mal,
and thus H\"older continuous with exponent $1/K$, we deduce that 
\begin{equation} \label{1y}
\frac{1}{K} \dim J(f_\lambda)\leq \dim J(f_\kappa)\leq K\dim J(f_\lambda).
\end{equation} 
It follows from~\eqref{1r} that $K\to 1$ as $\kappa\to\lambda$.
Thus we deduce from~\eqref{1y} that $\dim  J(f_\kappa)\to \dim J(f_\lambda)$  as $\kappa\to\lambda$.
Hence $\lambda\mapsto \dim J(f_\lambda)$ is continuous. \qed

\begin{remark}\label{remark2}
A celebrated result of Astala~\cite[Corollary~1.3]{Astala1994} says that~\eqref{1y} can
be improved to 
\begin{equation} \label{1x}
\frac{1}{K} \left( \frac{1}{\dim J(f_\lambda)}-\frac12\right)
\leq 
\frac{1}{\dim J(f_\kappa)}-\frac12
\leq 
K \left( \frac{1}{\dim J(f_\lambda)}-\frac12\right).
\end{equation} 
For our purposes, however, the weaker and simpler estimate~\eqref{1y} suffices.
\end{remark}

\noindent
W. Bergweiler: Mathematisches Seminar,
Christian--Albrechts--Universit\"at zu Kiel,
Ludewig--Meyn--Stra{\ss}e 4, 
24098 Kiel, 
Germany\\
Email: bergweiler@math.uni-kiel.de

\medskip

\noindent
W. Cui: Centre for Mathematical Sciences, Lund University, Box 118, 22 100 Lund, Sweden\\
E-mail: weiwei.cui@math.lth.se

\end{document}